\documentclass[runningheads,12pt]{article}
\usepackage{authblk}

\usepackage{amsmath,amsthm}
\usepackage{hyperref}
\usepackage{nicefrac}

\usepackage{color}

\usepackage[margin=1in]{geometry}

\usepackage{pxfonts}
\usepackage[footnotesize]{caption}

\usepackage{pstricks}
\psset{linewidth=0.5pt,arrowlength=2.8}

\newtheorem{theorem}{Theorem}
\newtheorem{lemma}[theorem]{Lemma}
\newtheorem{corollary}[theorem]{Corollary}

\newcommand{\os}{\overline{\sigma}}
\newcommand{\TT}{{\cal T}}

\begin{document}

\title{On the minimal teaching sets of two-dimensional threshold functions}
\author[1]{Max A. Alekseyev}
\author[2]{Marina G. Basova}
\author[2,3]{Nikolai Yu. Zolotykh}
\affil[1]{George Washington University, Washington, DC, USA}
\affil[2]{Lobachevsky State University, Nizhni Novgorod, Russia}
\affil[3]{Higher School of Economics, Nizhni Novgorod, Russia}

\maketitle

\begin{abstract}
It is known that a minimal teaching set of any threshold function on the two-dimensional rectangular grid consists of $3$ or $4$ points. 
We derive exact formulae for the numbers of functions corresponding to these values 
and further refine them in the case of a minimal teaching set of size $3$.
We also prove that the average cardinality of the minimal teaching sets of threshold functions is asymptotically $\nicefrac{7}{2}$.

We further present corollaries of these results concerning some special arrangements of lines in the plane.
\end{abstract}






\thispagestyle{empty}

\newpage

\section{Introduction}

A mapping of the domain $\{0,1,\dots,m-1\}\times\{0,1,\dots,n-1\}$ into $\{0,1\}$, 
is called a {\em threshold function} iff there exists a line separating the sets of its ones and zeros
(points at which the function takes values $1$ and $0$, respectively).

A \emph{teaching set} of a threshold function $g$ is a subset of its domain 
such that the values of $g$ at the points from this set uniquely define the function.
A teaching set of $g$ is called {\em minimal} (or {\em irreducible}) iff no its proper subset is teaching for $g$.
It is known (e.g., see \cite{ABS1995,ShevchenkoZolotykh1998,ZolotykhShevchenko1999}) 
that the minimal teaching set $T(g)$ for any threshold function $g$ is unique
(this also holds for threshold functions in higher dimensions).
In \cite{ShevchenkoZolotykh1998,Zolotykh2001} it is proved that $T(g)$ consists of $3$ or $4$ points if $m\ge 2$, $n\ge 2$.
Here we derive exact formulae for the number of threshold functions corresponding to these values.
In the case when $T(g)$ consists of $3$ points, we further compute the number of threshold functions $g$ with the specified number (one or two) of zeros 
in $T(g)$, answering the question posed in \cite{Alekseyev2010}.
We also prove that the average cardinality of the minimal teaching set is asymptotically $\nicefrac{7}{2}$ as $m,n\to\infty$. 
Finally, we discuss some corollaries of this result concerning special arrangements of lines in the plane.

We remark that the cardinality of the minimal teaching set of threshold functions defined on {\em many\,}-dimensional 
grid is studied in a number of publications (e.g., see the bibliography in \cite{Zolotykh2008}). 
In particular, the bounds of the average cardinality of the minimal teaching set of threshold functions are given in \cite{ABS1995,VirovlyanskayaZolotykh2003}.

\section{Preliminary definitions and results}

Let $m\ge 2$, $n\ge 2$, $E_m=\{0,1,\dots,m-1\}$, $g:~E_m\times E_n \to \{0,1\}$.
We denote by $M_0(g)$, $M_1(g)$ the set of zeros and the set of ones of $g$, respectively, that is
$$
M_{\nu}(g) = \{x=(x_1, x_2)\in E_m\times E_n:~f(x)=\nu\}
\quad
(\nu \in \{0,1\}).
$$
The function $g$ is called {\em threshold} iff there exist real numbers $a_0$, $a_1$, $a_2$ such that 
\begin{equation}
M_0(g) = \{(x_1,x_2)\in E_m\times E_n\ :\ a_1 x_1 + a_2 x_2 \le a_0 \}.
\label{eq:M0}
\end{equation}
Without loss of generality we can assume that the numbers $a_0$, $a_1$, $a_2$ are integer and $a_1$, $a_2$ are not zero simultaneously.
We call the line $a_1 x_1 + a_2 x_2 = a_0$ a {\em separation line} for the threshold function $g$.
Note that for any line $a_1 x_1 + a_2 x_2 = a_0$, there are two associated threshold functions:
a function $g$ defined by \eqref{eq:M0} holds and a function $h$ with $M_0(h) = \{(x_1,x_2)\in E_m\times E_n\ :\ a_1 x_1 + a_2 x_2 > a_0 \}$, which also satisfy
$g(x)+h(x)=1$ for any $x\in E_m\times E_n$.

Let $\TT(m,n)$ be the set of all threshold functions defined on $E_m\times E_n$ and $t(m,n)=|\TT(m,n)|$.

Denote
$$
f_q(m,n) = \sum_{\substack{-m<i<m\\ -n<j<n\\ \gcd(i,j)=q}} (m-|i|)(n-|j|).
$$

Two distinct points $p$, $p'$ in $E_m \times E_n$ are called {\em adjacent} 
if and only if the line segment $[p,p']$ contains no other point from $E_m \times E_n$. 

\begin{lemma}[\cite{AcketaZunic1991,KLB1990}] \label{lemma_adj} 
The number of ordered pairs $(p, p')$ of adjacent points in $E_m\times E_n$ equals $f_1(m,n)$.
\end{lemma}

\begin{lemma}[\cite{HaukkanenMerikoski2013}] \label{lemma_fqmn} 
If $m\le n$ then 
$$
f_q(m,n) = \frac{6}{\pi^2 q^2}(mn)^2 + O(mn^2).
$$
\end{lemma}


Denote by $l(m,n)$ the number of lines passing through at least two points in $E_m\times E_n$.
\begin{lemma}[\cite{Mustonen2009,HaukkanenMerikoski2012}]
\begin{equation}
l(m,n) = \frac{1}{2}\left(f_1(m,n) - f_2(m,n) \right).
\label{eq:lmnf1f2}
\end{equation}
\end{lemma}

\begin{theorem}[\cite{KLB1990}] \label{th_tmn}
\begin{equation}
t(m,n) = f_1(m,n) + 2.
\label{eq:tmnfmn}
\end{equation}
\end{theorem}

\begin{corollary}
\begin{equation}
t(m,n) = 2 + \sum_{\substack{-m<i<m\\ -n<j<n\\ \gcd(i,j)=1}} (m-|i|)(n-|j|) = \frac{6}{\pi^2}(mn)^2 + O(mn^2),
\label{eq:tmnass}
\end{equation}
where asymptotics holds for $m\le n$.
\end{corollary}

We remark that the asymptotic behavior of $t(m,n)$ was studied by many authors (sometimes using different terminology), for example, 
\cite{KLB1990,AcketaZunic1991,Shevchenko1995,Alekseyev2010,Zunic2011,HaukkanenMerikoski2013}.

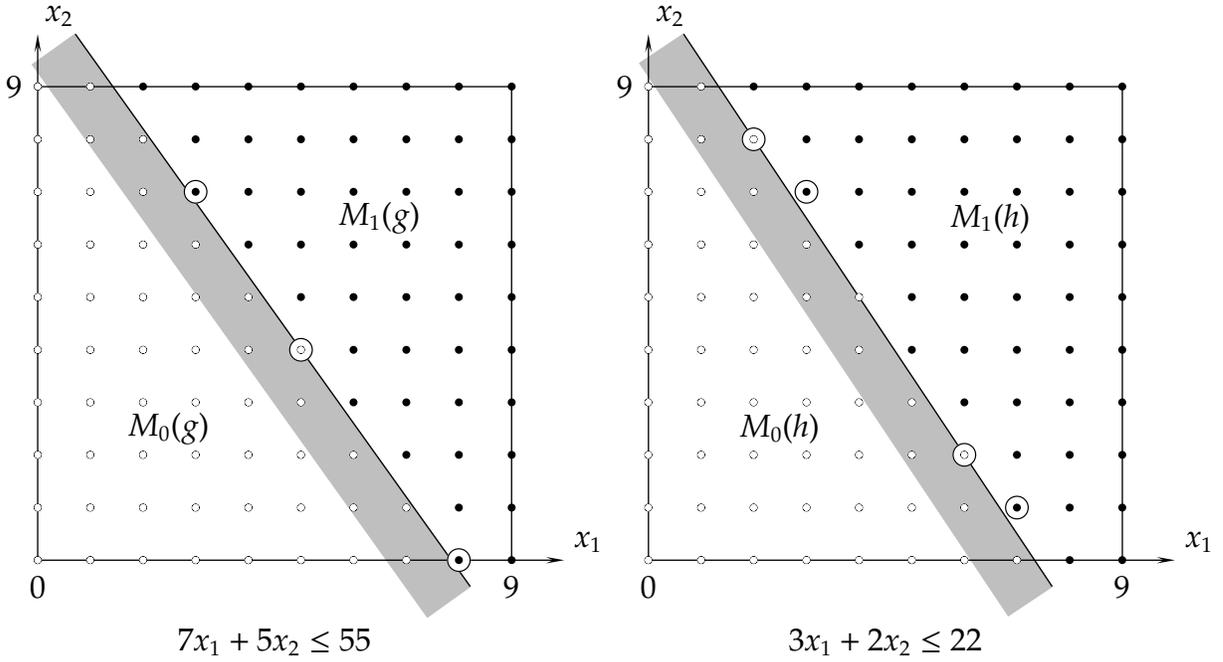
\begin{figure} 
\begin{center}

\begin{tabular}{lr}

\psset{xunit=.7cm,yunit=.7cm}
\begin{pspicture}(-1,-2)(10,10)

  \rput[t](4.5,-1.3){$7x_1 + 5x_2 \le 55$}

  \pspolygon[fillstyle=solid,fillcolor=lightgray,linecolor=lightgray](8.21429,-0.5)(0.7143,10)(-0.0994,9.4188)(7.4006,-1.0812)
  \psline(8.21429,-0.5)(0.7143,10)

  \psline[arrows=->](10,0)
  \psline[arrows=->](0,10)
  \rput[lb](10.2,0.15){$x_1$}
  \rput[lb](0.15,10.15){$x_2$}
  \rput[t](0,-0.3){$0$}
  \rput[t](9,-0.3){$9$}
  \rput[r](-0.3,9){$9$}
  \psline(0,9)(9,9)(9,0)
  \rput(2.5,2.5){$M_0(g)$}
  \rput(6.5,6.5){$M_1(g)$}

  \psdots[dotstyle=o,dotscale=3](5,4)
  \psdots[dotstyle=o,dotscale=3](8,0)(3,7)

  \psdots[dotstyle=o](0,0)\psdots[dotstyle=o](1,0)\psdots[dotstyle=o](2,0)\psdots[dotstyle=o](3,0)\psdots[dotstyle=o](4,0)\psdots[dotstyle=o](5,0) 
  \psdots[dotstyle=o](0,1)\psdots[dotstyle=o](1,1)\psdots[dotstyle=o](2,1)\psdots[dotstyle=o](3,1)\psdots[dotstyle=o](4,1)\psdots[dotstyle=o](5,1) 
  \psdots[dotstyle=o](0,2)\psdots[dotstyle=o](1,2)\psdots[dotstyle=o](2,2)\psdots[dotstyle=o](3,2)\psdots[dotstyle=o](4,2)\psdots[dotstyle=o](5,2) 
  \psdots[dotstyle=o](0,3)\psdots[dotstyle=o](1,3)\psdots[dotstyle=o](2,3)\psdots[dotstyle=o](3,3)\psdots[dotstyle=o](4,3)\psdots[dotstyle=o](5,3) 
  \psdots[dotstyle=o](0,4)\psdots[dotstyle=o](1,4)\psdots[dotstyle=o](2,4)\psdots[dotstyle=o](3,4)\psdots[dotstyle=o](4,4)\psdots[dotstyle=o](5,4) 
  \psdots[dotstyle=o](0,5)\psdots[dotstyle=o](1,5)\psdots[dotstyle=o](2,5)\psdots[dotstyle=o](3,5)\psdots[dotstyle=o](4,5)\psdots[          ](5,5) 
  \psdots[dotstyle=o](0,6)\psdots[dotstyle=o](1,6)\psdots[dotstyle=o](2,6)\psdots[dotstyle=o](3,6)\psdots[          ](4,6)\psdots[          ](5,6) 
  \psdots[dotstyle=o](0,7)\psdots[dotstyle=o](1,7)\psdots[dotstyle=o](2,7)\psdots[          ](3,7)\psdots[          ](4,7)\psdots[          ](5,7) 
  \psdots[dotstyle=o](0,8)\psdots[dotstyle=o](1,8)\psdots[dotstyle=o](2,8)\psdots[          ](3,8)\psdots[          ](4,8)\psdots[          ](5,8) 
  \psdots[dotstyle=o](0,9)\psdots[dotstyle=o](1,9)\psdots[          ](2,9)\psdots[          ](3,9)\psdots[          ](4,9)\psdots[          ](5,9)

  \psdots[dotstyle=o](6,0)\psdots[dotstyle=o](7,0)\psdots[          ](8,0)\psdots[          ](9,0) 
  \psdots[dotstyle=o](6,1)\psdots[dotstyle=o](7,1)\psdots[          ](8,1)\psdots[          ](9,1) 
  \psdots[dotstyle=o](6,2)\psdots[          ](7,2)\psdots[          ](8,2)\psdots[          ](9,2) 
  \psdots[          ](6,3)\psdots[          ](7,3)\psdots[          ](8,3)\psdots[          ](9,3) 
  \psdots[          ](6,4)\psdots[          ](7,4)\psdots[          ](8,4)\psdots[          ](9,4) 
  \psdots[          ](6,5)\psdots[          ](7,5)\psdots[          ](8,5)\psdots[          ](9,5) 
  \psdots[          ](6,6)\psdots[          ](7,6)\psdots[          ](8,6)\psdots[          ](9,6) 
  \psdots[          ](6,7)\psdots[          ](7,7)\psdots[          ](8,7)\psdots[          ](9,7) 
  \psdots[          ](6,8)\psdots[          ](7,8)\psdots[          ](8,8)\psdots[          ](9,8) 
  \psdots[          ](6,9)\psdots[          ](7,9)\psdots[          ](8,9)\psdots[          ](9,9) 


\end{pspicture}


&


\psset{xunit=.7cm,yunit=.7cm}
\begin{pspicture}(-1,-2)(10,10)

  \rput[t](4.5,-1.3){$3x_1 + 2x_2 \le 22$}

  \pspolygon[fillstyle=solid,fillcolor=lightgray,linecolor=lightgray](7.66667,-0.5)(0.666667,10)(-0.1654,9.4453)(6.8346,-1.0547)
  \psline(7.66667,-0.5)(0.666667,10) 

  \psline[arrows=->](10,0)
  \psline[arrows=->](0,10)
  \rput[lb](10.2,0.15){$x_1$}
  \rput[lb](0.15,10.15){$x_2$}
  \rput[t](0,-0.3){$0$}
  \rput[t](9,-0.3){$9$}
  \rput[r](-0.3,9){$9$}
  \psline(0,9)(9,9)(9,0)
  \rput(2.5,2.5){$M_0(h)$}
  \rput(6.5,6.5){$M_1(h)$}

  \psdots[dotstyle=o,dotscale=3](6,2)(2,8)
  \psdots[dotstyle=o,dotscale=3](7,1)(3,7)

  \psdots[dotstyle=o](0,0)\psdots[dotstyle=o](1,0)\psdots[dotstyle=o](2,0)\psdots[dotstyle=o](3,0)\psdots[dotstyle=o](4,0)\psdots[dotstyle=o](5,0) 
  \psdots[dotstyle=o](0,1)\psdots[dotstyle=o](1,1)\psdots[dotstyle=o](2,1)\psdots[dotstyle=o](3,1)\psdots[dotstyle=o](4,1)\psdots[dotstyle=o](5,1) 
  \psdots[dotstyle=o](0,2)\psdots[dotstyle=o](1,2)\psdots[dotstyle=o](2,2)\psdots[dotstyle=o](3,2)\psdots[dotstyle=o](4,2)\psdots[dotstyle=o](5,2) 
  \psdots[dotstyle=o](0,3)\psdots[dotstyle=o](1,3)\psdots[dotstyle=o](2,3)\psdots[dotstyle=o](3,3)\psdots[dotstyle=o](4,3)\psdots[dotstyle=o](5,3) 
  \psdots[dotstyle=o](0,4)\psdots[dotstyle=o](1,4)\psdots[dotstyle=o](2,4)\psdots[dotstyle=o](3,4)\psdots[dotstyle=o](4,4)\psdots[          ](5,4) 
  \psdots[dotstyle=o](0,5)\psdots[dotstyle=o](1,5)\psdots[dotstyle=o](2,5)\psdots[dotstyle=o](3,5)\psdots[dotstyle=o](4,5)\psdots[          ](5,5) 
  \psdots[dotstyle=o](0,6)\psdots[dotstyle=o](1,6)\psdots[dotstyle=o](2,6)\psdots[dotstyle=o](3,6)\psdots[          ](4,6)\psdots[          ](5,6) 
  \psdots[dotstyle=o](0,7)\psdots[dotstyle=o](1,7)\psdots[dotstyle=o](2,7)\psdots[          ](3,7)\psdots[          ](4,7)\psdots[          ](5,7) 
  \psdots[dotstyle=o](0,8)\psdots[dotstyle=o](1,8)\psdots[dotstyle=o](2,8)\psdots[          ](3,8)\psdots[          ](4,8)\psdots[          ](5,8) 
  \psdots[dotstyle=o](0,9)\psdots[dotstyle=o](1,9)\psdots[          ](2,9)\psdots[          ](3,9)\psdots[          ](4,9)\psdots[          ](5,9)

  \psdots[dotstyle=o](6,0)\psdots[dotstyle=o](7,0)\psdots[          ](8,0)\psdots[          ](9,0) 
  \psdots[dotstyle=o](6,1)\psdots[          ](7,1)\psdots[          ](8,1)\psdots[          ](9,1) 
  \psdots[dotstyle=o](6,2)\psdots[          ](7,2)\psdots[          ](8,2)\psdots[          ](9,2) 
  \psdots[          ](6,3)\psdots[          ](7,3)\psdots[          ](8,3)\psdots[          ](9,3) 
  \psdots[          ](6,4)\psdots[          ](7,4)\psdots[          ](8,4)\psdots[          ](9,4) 
  \psdots[          ](6,5)\psdots[          ](7,5)\psdots[          ](8,5)\psdots[          ](9,5) 
  \psdots[          ](6,6)\psdots[          ](7,6)\psdots[          ](8,6)\psdots[          ](9,6) 
  \psdots[          ](6,7)\psdots[          ](7,7)\psdots[          ](8,7)\psdots[          ](9,7) 
  \psdots[          ](6,8)\psdots[          ](7,8)\psdots[          ](8,8)\psdots[          ](9,8) 
  \psdots[          ](6,9)\psdots[          ](7,9)\psdots[          ](8,9)\psdots[          ](9,9)

\end{pspicture}

\end{tabular}

\end{center}
\caption{Threshold functions $g,h$ on $E_{10}\times E_{10}$ with $|T(g)|=3$ (left panel) and $|T(h)|=4$ (right panel) and their separation lines.
The empty dots represent zeros of the function, while the circled dots form its minimal teaching set.}\label{fig_exTf}
\end{figure}


\section{Cardinality of minimal teaching sets}

A set $T\subseteq E_m\times E_n$ is called {\em teaching} for $g\in\TT(m,n)$
iff for any other function $h\in\TT(m,n)$ there exists $x\in T$ such that $g(x)\ne h(x)$.
A teaching set of $g$ is called {\em minimal} or {\em irreducible} iff no its proper subset is teaching for $g$.
A point $x \in E_m\times E_n$ is called {\em essential} for $g\in\TT(m,n)$ iff there exists $h\in\TT(m,n)$
such that $g(x)\ne h(x)$ and $g(y)=h(y)$ for all $y\ne x$.

\begin{lemma}[\cite{ABS1995,ShevchenkoZolotykh1998,ZolotykhShevchenko1999}] \label{lemma_Tunique} 
For any function $g\in\TT(m,n)$, the minimal teaching set of $g$ is unique and consists of all essential points of $g$.
\end{lemma}
We remark that Lemma~\ref{lemma_Tunique} also holds for threshold functions defined on many-dimensional grid.

Denote by $T(g)$ the minimal teaching set of $g$.
In~\cite{ShevchenkoZolotykh1998,Zolotykh2001} it is shown that $|T(g)|\in\{3,4\}$ for any function $g\in\TT(m,n)$
(see examples in Fig.~\ref{fig_exTf}).

We define the {\em average cardinality} of the minimal teaching sets as
\begin{equation}
\os(m,n) = \frac{1}{t(m,n)} \sum_{g\in\TT(m,n)} |T(g)|. 
\label{eq:osdef}
\end{equation}

\begin{theorem} \label{osmn} 
\begin{equation}
\os(m,n) = \frac{4\cdot f_1(m,n)-2\cdot f_2(m,n)}{f_1(m,n)+2}.
\label{eq:osfff1}
\end{equation}
\end{theorem}

\begin{proof}
Denote by $h(m,n,i,j)$ the number of threshold functions defined on the domain $E_m\times E_n$ such 
that the point $(i,j)$ is essential. By Lemma \ref{lemma_Tunique},
\begin{equation}
\os(m,n) = \frac{1}{t(m,n)} \sum_{g\in \TT(m,n)} |T(g)| = \frac{1}{t(m,n)} \sum_{i=0}^{m-1}\sum_{j=0}^{n-1} h(m,n,i,j).
\label{eq:osmnh}
\end{equation}

\begin{figure} 
\begin{center}

\psset{xunit=1.7cm,yunit=1.7cm}
\begin{pspicture}(-0.75,-1)(4,4)

  \psline[linewidth=1pt](1.35,-0.4)(0.4,3.4)

  \psline[arrows=->](3.5,0)
  \psline[arrows=->](0,3.75)
  \rput(3.65,0.15){$x_1$}
  \rput[r](0,3.8){$x_2$~}
  \psline(-0.2,1)(3.2,1)
  \psline(0,3)(3,3)
  \psline(1,-0.2)(1,3.2)
  \psline(3,0)(3,3)
  \psline(-0.2,-0.2)(3.2,3.2)
  \psline(-0.2,2.2)(2.2,-0.2)
  \psline(-0.1,3.2)(1.6,-0.2)
  \psline(2.1,3.2)(0.4,-0.2)
  \psline(3.2,2.1)(-0.2,0.4)
  \psline(3.2,-0.1)(-0.2,1.6)
 
  \psdots[dotstyle=o](0,3) \psdots[          ](1,3) \psdots[ ](2,3) \psdots[ ](3,3) 
  \psdots[dotstyle=o](0,2) \psdots[          ](1,2) \psdots[ ](2,2) \psdots[ ](3,2) 
  \psdots[dotstyle=o](0,1) \psdots[dotstyle=o](1,1) \psdots[ ](2,1) \psdots[ ](3,1) 
  \psdots[dotstyle=o](0,0) \psdots[dotstyle=o](1,0) \psdots[ ](2,0) \psdots[ ](3,0) 

  \rput(0,-0.2){$0$}
  \rput[r](1,-0.2){$1~$}
  \rput(2,-0.2){$2$}
  \rput(3,-0.2){$3$}
  \rput(-0.15,0.85){$1$}
  \rput(-0.15,1.95){$2$}
  \rput(-0.15,3){$3$}

\rput(0.6,3.4){$\ell'$}

%

\end{pspicture}

\caption{The lines forming $L(4,4,1,1)$ with $l(4,4,1,1) = 8$ and a line $\ell'\notin L(4,4,1,1)$
defining two threshold functions.
}
\label{Fig2}

\end{center}
\end{figure}
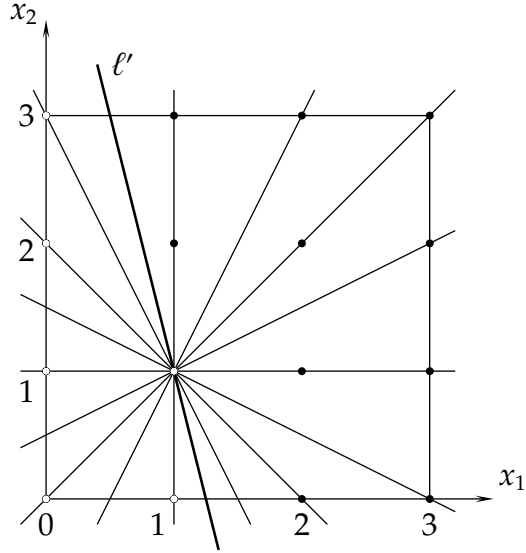

Let us consider all lines containing the point $(i,j)$ and at least one other point from $E_m\times E_n$.
Denote by $L(m,n,i,j)$ the set of all such lines (see Fig.~\ref{Fig2}) and
let $l(m,n,i,j) = |L(m,n,i,j)|$.
The lines from $L(m,n,i,j)$ partition the plane into $2\cdot l(m,n,i,j)$ sectors.

It is easy to see that any line $\ell'\notin L(m,n,i,j)$ containing the point $(i,j)$ is associated with two threshold functions, 
for which the point $(i,j)$ is essential and a zero (Fig.~\ref{Fig2}).
It is clear that these functions are uniquely determined by the two lines from $L(m,n,i,j)$ adjacent to $\ell'$.
Moreover, there is no other function in $\TT(m,n)$, for which the point $(i,j)$ is essential and a zero.

Thus, we have $2\cdot l(m,n,i,j)$ threshold functions for which the point $(i,j)$ is essential and a zero.
We have the same number of functions for which the point $(i,j)$ is essential and not a zero (i.e., the value at $(i,j)$ is $1$).
Hence $h(m,n,i,j) = 4\cdot l(m,n,i,j)$.
Now from \eqref{eq:osmnh} we get
\begin{equation}
\label{eq:osmnh2}
\os(m,n) = \frac{4}{t(m,n)} \sum_{i=0}^{m-1}\sum_{j=0}^{n-1} l(m,n,i,j).
\end{equation}

Let $L(m,n)$ be the set of all lines passing through at least two points from $E_m\times E_n$ and $l(m,n)=|L(m,n)|$.
For a line $\ell \in L(m,n)$, we denote by $z(m,n,\ell)$ the number of points in $E_m\times E_n$ belonging to $\ell$.
Now from \eqref{eq:osmnh2} we get
$$
\os(m,n) 
= \frac{4}{t(m,n)} \sum_{\ell\in L(m,n)} z(m,n,\ell).
$$
We remark that if $\ell$ contains exactly $k$ points from $E_m\times E_n$, then they form $k-1$ adjacent pairs of points.
Then by Lemma~\ref{lemma_adj},
$$
\os(m,n) = 4\cdot \frac{l(m,n)+\frac{1}{2}f_1(m,n)}{t(m,n)}.
$$
Substituting here the expressions \eqref{eq:lmnf1f2} and \eqref{eq:tmnfmn}, we obtain \eqref{eq:osfff1}.
\end{proof}

Theorem~\ref{osmn} and Lemma~\ref{lemma_fqmn} imply the following asymptotics for $\os(m,n)$.

\begin{corollary}\label{cor72} For $m\le n$,
$$
\os(m,n) = \frac{7}{2} + O\left(\frac{1}{m}\right).
$$
\end{corollary}


Denote by $t_{\kappa}(m,n)$ the number of functions $g\in\TT(m,n)$ such that $|T(g)|=\kappa$ ($\kappa\in\{3,4\}$).
Corollary~\ref{cor72} implies that each of the quantities $t_3(m,n)$ and $t_4(m,n)$ is asymptotically $\frac{1}{2}t(m,n)$.
Below we give exact formulae for $t_3(m,n)$ and $t_4(m,n)$.  

\begin{theorem}\label{th_t3t4}
\begin{eqnarray*}
t_3(m,n) &=& 2\cdot f_2(m,n) + 8,\\
t_4(m,n) &=& f_1(m,n) - 2\cdot f_2(m,n) - 6.
\end{eqnarray*}
\end{theorem}
\begin{proof}
From the conditions
$$
t_3(m,n) + t_4(m,n) = t(m,n), \qquad \os(m,n) = \frac{3\cdot t_3(m,n) + 4\cdot t_4(m,n)}{t(m,n)},
$$
we obtain
$$
t_3(m,n) = \left(4 - \os(m,n)\right)\cdot t(m,n), \qquad t_4(m,n) = \left(\os(m,n) - 3\right)\cdot t(m,n).
$$
Substitution of the expressions \eqref{eq:tmnfmn} and \eqref{eq:osfff1} here completes the proof.
\end{proof}


\section{Minimal teaching sets of size $3$}

It is known that a minimum teaching set of a non-constant threshold function $g$ with $|T(g)|=4$ always contains two zeros~\cite{Zolotykh2001}.
The threshold functions $g$ with $|T(g)|=3$ can be classified depending on whether their minimal teaching sets contain one or two zeros.
The question of counting threshold functions in these classes was posed and partially answered in~\cite{Alekseyev2010}.
Below we give a complete answer to this question.

Let $u_{\nu,\kappa}(m,n)$ be the number of threshold functions $g$ on $E_m\times E_n$ with $|T(g)|=3$ such that $g(0,0)=\nu$ ($\nu\in\{0,1\}$) and 
a minimum teaching set of $g$ consists of $\kappa$ zeros and $3-\kappa$ ones ($\kappa\in\{1,2\}$). 
The number $u_{0,1}(m,n)$ enumerates so-called \emph{unstable threhold functions}~\cite{Alekseyev2010}:

\begin{lemma}[\cite{Alekseyev2010}] \label{lemma:unstable}
$$ u_{0,1}(m,n) = 2\cdot f_1(\tfrac{m}{2},\tfrac{n}{2})+2 - s(m,n), $$
where
$$s(m,n) = \sum_{\substack{0<i<m\\ 0<j<n\\ \gcd(i,j)=1}} 1,$$
which can be also expressed as
$$s(m,n) = \frac{1}{4}\left( f_1(m,n) - f_1(m-1,n) - f_1(m,n-1) + f_1(m-1,n-1) \right) - 1.$$
\end{lemma}

\begin{theorem}
$$
\begin{array}{lllll}
u_{0,1}(m,n) &=& u_{1,2}(m,n) &=& 2\cdot f_1(\tfrac{m}{2},\tfrac{n}{2})+2 - s(m,n),\\
u_{0,2}(m,n) &=& u_{2,1}(m,n) &=& f_2(m,n) + 2 - 2\cdot f_1(\tfrac{m}{2},\tfrac{n}{2}) + s(m,n).
\end{array}
$$
Furthermore, for any $\nu\in\{0,1\}$, $\kappa\in\{1,2\}$, and $m\leq n$, we have
$$u_{\nu,\kappa}(m,n) = \frac{3}{4\pi^2}(mn)^2 + O(mn^2).$$
\end{theorem}

\begin{proof} Trivially, we have $u_{0,1}(m,n) + u_{0,2}(m,n) + u_{1,1}(m,n) + u_{1,2}(m,n) = t_3(m,n)$. 
Furthermore, the bijection $g\mapsto 1-g$ implies that $u_{\nu,\kappa}(m,n) = u_{1-\nu,3-\kappa}(m,n)$ and thus we can focus only on the case $\nu=0$, 
for which $u_{0,1}(m,n) + u_{0,2}(m,n) = \tfrac{1}{2}t_3(m,n)$.
The exact formulae now easily follow from Theorem~\ref{th_t3t4} and Lemma~\ref{lemma:unstable}. To obtain the asymptotics, we use Lemma~\ref{lemma_fqmn} 
and notice that $s(m,n) = O(mn)$. 
\end{proof}

\section{Special arrangements of lines}

\subsection{Partition of the plane}

For a threshold function $g\in\TT(m,n)$, all vectors $(a_0, a_1, a_2)\in\mathbb{R}^3$, where  $a_0, a_1, a_2$ are coefficients of its separating line,
form a polyhedral cone defined by the following system of linear inequalities:
\begin{equation}
\begin{cases}
	a_1 x_1 + a_2 x_2 \le a_0 & \mbox{for each $(x_1,x_2)\in M_0(g)$}, \\
	a_1 x_1 + a_2 x_2 >   a_0 & \mbox{for each $(x_1,x_2)\in M_1(g)$}. 
\end{cases}
\label{eq:Kf}
\end{equation}
Furthermore, a minimal teaching set $T(g)$ consists of the points from $E_m\times E_n$
that correspond to irredundant inequalities in \eqref{eq:Kf}.
Thus, there is a bijection of the set $\TT(m,n)$ into the set of cones in the partition of the space of parameters $\{ (a_0, a_1, a_2)\in\mathbb{R}^3 \}$
by all the planes $a_1 x_1 + a_2 x_2 = a_0$, where $(x_1,x_2)\in E_m\times E_n\setminus \{(0,0)\}$. 
Moreover, the planes that form the boundary of each of these cones correspond to the points in $T(g)$.
This construction is well known in the threshold logic~\cite{ShevchenkoZolotykh1998,ZolotykhShevchenko1999}.

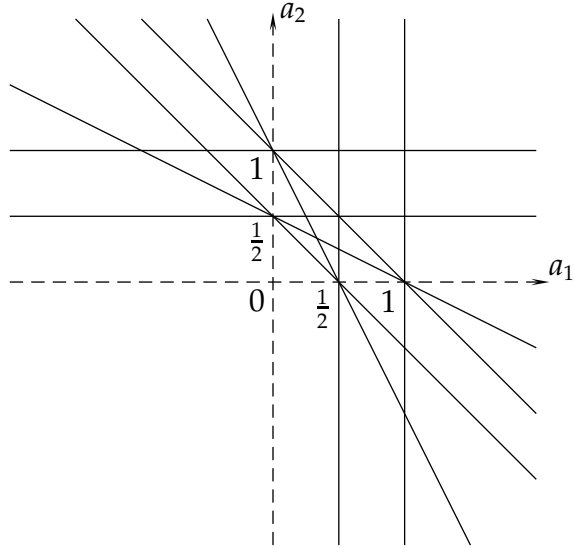
\begin{figure} 
\begin{center}

\psset{xunit=1.75cm,yunit=1.75cm}
\begin{pspicture}(-2,-2.1)(2,2)

  \psline[arrows=->,linestyle=dashed](-2,0)(2.1,0)
  \psline[arrows=->,linestyle=dashed](0,-2)(0,2.05)
  \rput[lb](2.1,0.015){$a_1$}
  \rput[l](0,2.05){~$a_2$~}
  
\psline(-1,2)(2,-1)  
\psline(-2,1)(2,1)  
\psline(-0.5,2)(1.5,-2)  

\psline(-2,0.5)(2,0.5)
\psline(-2,1.5)(2,-0.5)  
\psline(-1.5,2)(2,-1.5)  

\psline(0.5,2)(0.5,-2)
\psline(1,2)(1,-2)

\rput[tr](1,-0.055){$1$~}
\rput[tr](0.5,-0.055){$\frac{1}{2}$~}
\rput[tr](0,-0.055){$0$~}

\rput[rt](0,0.95){$1$~}
\rput[rt](0,0.45){$\frac{1}{2}$~}
 

\end{pspicture}

\caption{The partition of the plane $\{ (a_1, a_2)\in\mathbb{R}^2 \}$ by the lines $a_1 x_1 + a_2 x_2 = 1$, where $(x_1,x_2)\in E_3\times E_3\setminus \{(0,0)\}$.
There are $c(3,3)=29$ cells, $c_3(3,3)=20$ generalized 3-polygons, $c_4(3,3)=9$ generalized 4-polygons, $e(3,3)=43$ edges, and $v(3,3)=15$ vertices in this partition.}
\label{Fig3}

\end{center}
\end{figure}

For $\nu\in\{0,1\}$, let $\TT_{\nu}(m,n) = \{g\in \TT(m,n):~ g(0,0)=\nu \}$.
The mapping $g \mapsto 1-g$ represents a bijection between the sets $\TT_0(m,n)$ and $\TT_1(m,n)$,
for which the teaching sets are invariant.

Without loss of generality, for $g\in\TT_0(m,n)$, we can assume that $a_0=1$.
In this case the points $(a_1, a_2)\in\mathbb{R}^2$ such that the line $a_1 x_1 + a_2 x_2 = 1$ is separating for $g$ represent the solutions
to the following system of linear inequalities:
$$
\begin{cases}
	a_1 x_1 + a_2 x_2 \le 1 & \mbox{for all $(x_1,x_2)\in M_0(g)$}, \\
	a_1 x_1 + a_2 x_2 >   1 & \mbox{for all $(x_1,x_2)\in M_1(g)$}.
\end{cases}
$$
So we have the bijection of the set $\TT_0(m,n)$ into the set of all polygonal cells in the partition of the plane $\{ (a_1, a_2)\in\mathbb{R}^2 \}$
by the lines $a_1 x_1 + a_2 x_2 = 1$, where $(x_1,x_2)\in E_m\times E_n\setminus \{(0,0)\}$ (see\,Fig.~\ref{Fig3}). 


Let $c(m,n)$, $e(m,n)$, $v(m,n)$ be respectively the number of cells, edges, 
and vertices (i.e., intersection points) in this partition.
We call a convex polygon (possibly unbounded) by a \emph{generalized $k$-polygon} if it can be obtained as the intersection of
an $k$-faced polyhedral cone with a plane that does not contain the cone vertex. 
Every unbounded generalized $k$-polygon representing a cell in the partition has two parallel or two non-parallel \emph{infinite edges} (i.e., rays). 
In the former case, the polygon has $k-1$ vertices; in the latter case, it has $k-2$ vertices.

\begin{theorem}\label{th_partition1}
The cells in the partition of the plane $\{ (a_1, a_2)\in\mathbb{R}^2 \}$ by the lines $a_1x_1 + a_2x_2 = 1$, where 
$(x_1,x_2)\in E_m\times E_n$, are only generalized $3$- and $4$-polygons. 
The number of such cells is, respectively, 
\begin{equation}\label{eq_v3v4}
\begin{array}{lllllll}
c_3(m,n) &=& \frac{1}{2}t_3(m,n) &=& f_2(m,n) + 4 &=& \frac{3}{2\pi^2} m^2n^2 + O(mn^2), \\
c_4(m,n) &=& \frac{1}{2}t_4(m,n) &=& \frac{1}{2}f_1(m,n) - f_2(m,n) - 3 &=& \frac{3}{2\pi^2} m^2n^2 + O(mn^2),
\end{array}
\end{equation}
which imply that
\begin{equation}\label{eq_v}
c(m,n) = c_3(m,n)+c_4(m,n) = \frac{1}{2} f_1(m,n) + 1 = \frac{3}{\pi^2} m^2n^2 + O(mn^2).
\end{equation}
Moreover, 
\begin{equation}
e(m,n) = f_1(m,n) -\frac{1}{2} f_2(m,n) - s(m,n) - 2 = \frac{21}{4\pi^2} m^2n^2 + O(mn^2),
\label{eq:edges_plane}
\end{equation}
\begin{equation}
v(m,n) = \frac{1}{2} f_1(m,n) -\frac{1}{2} f_2(m,n) - s(m,n) - 2 = \frac{9}{4\pi^2} m^2n^2 + O(mn^2),
\label{eq:points_plane}
\end{equation}
where $s(m,n)$ is defined in Lemma~\ref{lemma:unstable}.
Everywhere above the asymptotics holds for $m\le n$.
\end{theorem}

\begin{proof}
The number of cells is $c(m,n)=|\TT_0(m,n)| = \frac{1}{2} t(m,n)$.
Among them there are $c_3(m,n)=\frac{1}{2}t_3(m,n)$ generalized $3$-polygons and $c_4(m,n)=\frac{1}{2}t_4(m,n)$ generalized $4$-polygons,
since the mapping $g \mapsto 1-g$ is a bijection between $\TT_0(m,n)$ and $\TT_1(m,n)$,
for which the teaching sets are invariant.
To obtain \eqref{eq_v3v4} and \eqref{eq_v} it remains to apply Theorem \ref{th_t3t4} and Lemma \ref{lemma_fqmn}.

Let $v_{\infty}(m,n)$ be the number of \emph{infinite vertices} (i.e., families of parallel lines) in the partition, 
which also equals the total number of \emph{infinitely-distant edges} in the unbounded generalized polygons.
Each family of parallel lines is defined by the slope, which may be zero (for the horizonal lines), infinite (for the vertical lines), 
or represent an irreducible fraction $\tfrac{p}{q}$ with $1\leq p\leq m-1$ and $1\leq q\leq n-1$. Since the number of such fractions is given by $s(m,n)$,
we have $v_{\infty}(m,n) = s(m,n) + 2$.


Since $3c_3(m,n) + 4c_4(m,n)$ gives twice the number of edges including infinitely-distant ones, we have
$
3c_3(m,n) + 4c_4(m,n) = 2e(m, n) + 2v_{\infty}(m,n),
$
which together with \eqref{eq_v3v4} implies \eqref{eq:edges_plane}.

To find the number of vertices, we apply Euler characteristic for the plane: 
$
v(m,n) - e(m,n) + c(m,n) = 1
$
and use \eqref{eq_v} and \eqref{eq:edges_plane} to obtain \eqref{eq:points_plane}.
\end{proof}


\subsection{Partition of the triangle}

Now we consider the triangle with vertices $(0,0)$, $(1,0)$, ($0,1)$ in the plane 
$\{ (a_1, a_2)\in\mathbb{R}^2 \}$ and its partition by the lines $a_1x_1 + a_2x_2 = 1$, where 
$(x_1,x_2)\in\{1,2,\dots,m\}\times\{1,2,\dots,n\}$ (Fig.\,\ref{Fig4}).
Let $c^{\triangle}(m,n)$, $e^{\triangle}(m,n)$, $v^{\triangle}(m,n)$ be respectively
the number of cells, edges, and vertices in this partition.

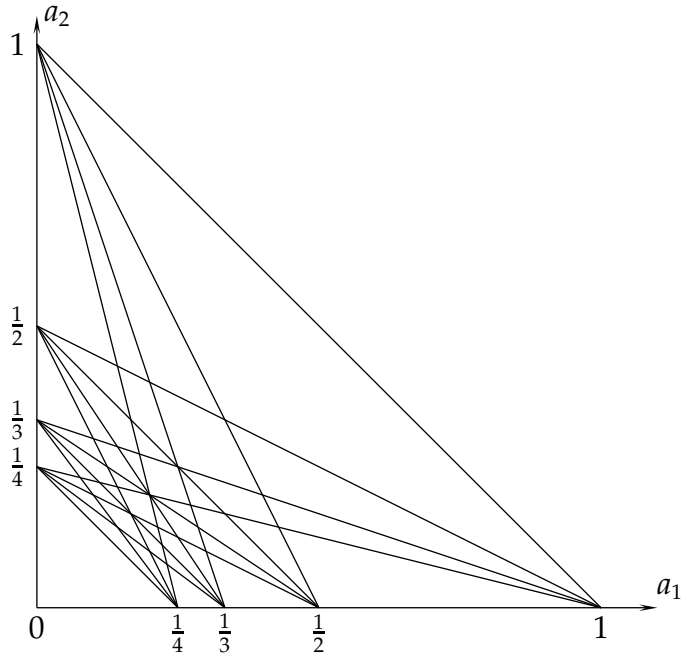
\begin{figure} 
\begin{center}

\psset{xunit=7.5cm,yunit=7.5cm}
\begin{pspicture}(0,-0.1)(1,1)

  \psline[arrows=->](1.1,0)
  \psline[arrows=->](0,1.05)
  \rput[lb](1.1,0.015){$a_1$}
  \rput[l](0,1.05){~$a_2$~}

\psline(0,1.00000)(1,0)\psline(0,1.00000)(0.5,0)\psline(0,1.00000)(0.333333,0)\psline(0,1.00000)(0.25,0)
\psline(0,0.50000)(1,0)\psline(0,0.50000)(0.5,0)\psline(0,0.50000)(0.333333,0)\psline(0,0.50000)(0.25,0)
\psline(0,0.33333)(1,0)\psline(0,0.33333)(0.5,0)\psline(0,0.33333)(0.333333,0)\psline(0,0.33333)(0.25,0)
\psline(0,0.25000)(1,0)\psline(0,0.25000)(0.5,0)\psline(0,0.25000)(0.333333,0)\psline(0,0.25000)(0.25,0)

\rput[tc](1,-0.035){$1$}
\rput[tc](0.5,-0.045){$\frac{1}{2}$}
\rput[tc](0.33333,-0.045){$\frac{1}{3}$}
\rput[tc](0.25,-0.045){$\frac{1}{4}$}
\rput[tc](0,-0.035){$0$}
\rput[r](-0.02,1){$1$}
\rput[r](-0.02,0.5){$\frac{1}{2}$}
\rput[r](-0.02,0.33333){$\frac{1}{3}$}
\rput[r](-0.02,0.25){$\frac{1}{4}$}
 

\end{pspicture}

\caption{The partition of the triangle by the lines $a_1 x_1 + a_2 x_2 = 1$, where $(x_1,x_2)\in \{1,2,3,4\}\times\{1,2,3,4\}$.
There are $c^{\triangle}(4,4)=47$ cells, $c^{\triangle}_3(4,4)=33$ triangles, $c^{\triangle}_4(4,4)=14$ quadrilaterals, $e^{\triangle}(4,4)=82$ edges,
and $v^{\triangle}(4,4)=36$ vertices in this partition.}
\label{Fig4}

\end{center}
\end{figure}

\begin{theorem}
Every cell in the partition of the triangle with vertices $(0,0)$, $(1,0)$, $(0,1)$ by the lines $a_1x_1 + a_2x_2 = 1$, where 
$(x_1,x_2)\in \{1,2,\dots,m\}\times\{1,2,\dots,n\}$, $m\ge 2$, $n\ge 2$, there is either a triangle or a quadrilateral. 
Moreover, the number of triangles and quadrilaterals is given by the following formulae:
$$
\begin{array}{lllll}
c^{\triangle}_3(m,n) & = & \frac{1}{2} f_2(m,n) + m + n + 1 & = & \frac{3}{4\pi^2} m^2n^2 + O(mn^2), \\
c^{\triangle}_4(m,n) & = & \frac{1}{4} f_1(m,n) - \frac{1}{2} f_2(m,n) - \frac{m}{2} - \frac{n}{2} - 1 & = & \frac{3}{4\pi^2} m^2n^2 + O(mn^2),
\end{array}
$$
which imply that
$$
c^{\triangle}(m,n) = c^{\triangle}_3(m,n) + c^{\triangle}_4(m,n) = \frac{1}{4} f_1(m,n) +\frac{m}{2} + \frac{n}{2} = \frac{3}{2\pi^2} m^2n^2 + O(mn^2).
$$
Moreover,
\begin{equation}
e^{\triangle}(m,n) = \frac{1}{2} f_1(m,n) -\frac{1}{4} f_2(m,n) + m + n = \frac{21}{8\pi^2} m^2n^2 + O(mn^2),
\label{eq:edges_triangle}
\end{equation}
\begin{equation}
v^{\triangle}(m,n) = \frac{1}{4} f_1(m,n) -\frac{1}{4} f_2(m,n) + \frac{m}{2} + \frac{n}{2} + 1 = \frac{9}{8\pi^2} m^2n^2 + O(mn^2).
\label{eq:points_triangle}
\end{equation}
Everywhere above the asymptotics holds for $m\le n$.
\end{theorem}

\begin{proof}
We notice that all cells incident to the axes are triangular and use Theorem~\ref{th_partition1} to obtain
$$
\begin{array}{lllll}
c^{\triangle}_3(m,n) & = & \frac{1}{4} t_3(m,n) + m + n - 1 & = & \frac{1}{2} f_2(m,n) + m + n + 1, \\
c^{\triangle}_4(m,n) & = & \frac{1}{4} t_4(m,n) - \frac{1}{2}(m + n - 1) & = & \frac{1}{4} f_1(m,n) - \frac{1}{2} f_2(m,n) - \frac{m}{2} - \frac{n}{2} - 1.
\end{array}
$$
The number $e^{\triangle}(m,n)$ of edges satisfies the equality:
$
2e^{\triangle}(m,n) = 3c^{\triangle}_3(m,n) + 4c^{\triangle}_4(m,n) + m + n + 1,
$
which implies \eqref{eq:edges_triangle}.
Now we use Euler characteristic for the triangle: 
$
v^{\triangle}(m,n) - e^{\triangle}(m,n) + c^{\triangle}(m,n) = 1
$
to obtain \eqref{eq:points_triangle}.
\end{proof}

\section*{Acknowledgments} 
The first author was supported by the National Science Foundation under the grant No. IIS-1253614.

\bibliographystyle{siam}
\bibliography{ave2d_en}

\end{document}